\documentclass[12pt,a4paper,reqno]{amsart}
\usepackage{amsmath}
\usepackage{amssymb}
\usepackage{amsthm}
\usepackage{mathrsfs}
\usepackage{enumitem}
\usepackage{etoolbox}
\usepackage{hyperref}
\hypersetup{
  colorlinks=true,
  citecolor=blue,
  linkcolor=blue,
  urlcolor=blue
}
\numberwithin{equation}{section}

\theoremstyle{plain}
\newtheorem{theorem}{Theorem}[section]
\newtheorem{lemma}{Lemma}[section]

\theoremstyle{definition}

\newtheorem{remark}{Remark}[section]
\makeatletter
\def\@settitle{%
  \begin{center}%
    \baselineskip14\p@
    {\Large \@title\par}%
  \end{center}%
}
\patchcmd{\@setauthors}{\MakeUppercase}{}{}{}
\patchcmd{\section}{\scshape}{}{}{}
\makeatother

\begin{document}
\title
[{Romanoff type theorems for polynomials over finite fields}]
{Romanoff type theorems for polynomials over finite fields}

\author
[Enci Wang]
{Enci Wang}

\address{Department of Mathematics, South China University of Technology, Guangzhou, Guangdong 510640, China}
\email{maecw@mail.scut.edu.cn}


\begin{abstract}
Given a polynomial \(g\) of degree \(\delta>0\) over a finite field, we study monic polynomials \(f\) of degree \(n\) that can be written as \(f=h+g^k\), where \(h\) is a monic irreducible polynomial of degree \(n\) and \(k\in\mathbb{N}\) with \(\delta k<n\). Following Erdős, we show that some \(f\) admit at least \(c\log n\) such representations. If \(\delta\rho_g<1\), where \(\rho_g=\varphi_q(g)/|g|\), then a positive proportion of degree-\(n\) polynomials admits at least two such representations. More generally, for sums of powers \(g^{\lfloor k_i^{r_i}\rfloor}\), we prove that a positive proportion of polynomials of degree \(n\) is representable whenever \(\sum_{i=1}^{t}r_i^{-1}\ge 1\), while certain residue classes modulo some irreducible polynomial contain no representable polynomial.
\end{abstract}

\maketitle

\section{Introduction}
In 1934, Romanoff \cite{romanoff1934} proved that the set of positive odd integers representable as
\[
p+2^{k},\qquad p\in\mathcal P,\ k\in\mathbb N,
\]
has positive lower asymptotic density, where \(\mathcal P\) denotes the set of all primes and \(\mathbb N\) the set of positive integers. 
This initiated the study of Romanoff-type problems in additive number theory \cite{nathanson1996}. 
Subsequent work refined the constant: in 2004, Chen and Sun \cite{chen2004} obtained the first explicit lower density \(0.0868\), improved in a number of papers (see, e.g.,  \cite{habsieger2006,lu2007,pintz2006}), and in 2018 Elsholtz and Schlage-Puchta \cite{elsholtz2018} established the currently best known lower bound \(0.107648\).

In 1950, Erd\H{o}s \cite{erdos1950} initiated two further directions. 
First, he proved that the number of representations of \(n\) in the form \(p+2^{k}\) is at least \(c\log\log n\) for infinitely many \(n\). 
Second, using the covering congruence systems he introduced for this purpose, he constructed an infinite arithmetic progression of odd integers admitting no representation as \(p+2^{k}\).

Generalizations with additional summands have also been investigated. 
In 1971, Crocker \cite{crocker1971} showed that infinitely many odd integers cannot be written as \(p+2^{m}+2^{n}\), a result considerably strengthened by Pan \cite{pan2011} in 2011. 
In 2014, Yang and Chen  \cite{yang2014} posed the problem of determining whether sums of a prime and several terms of the form \(2^{k^{2}}\) still have positive asymptotic density; in 2022, Ding \cite{ding2022} settled the case of two terms affirmatively. 
In 2024, Chen and Xu \cite{chen2024} established the theorem for the general \(t\)-term case. Define
\[
R_{2}(r_{1},\dots,r_{t})
=\left\{n=p+\sum_{i=1}^{t}2^{\lfloor k_{i}^{r_{i}}\rfloor}:
p\in\mathcal P,\ k_{i}\in\mathbb N\right\}.
\]
They proved that \(R_{2}(r_{1},\dots,r_{t})\) has asymptotic density zero if \(\sum_{i=1}^{t}r_{i}^{-1}<1\), and positive lower asymptotic density if \(\sum_{i=1}^{t}r_{i}^{-1}\ge1\), under the condition that at least one exponent \(r_{i}\) is an integer. 
They also constructed infinite arithmetic progressions of non-representable integers for several exponent tuples; whether such progressions exist for all tuples with \(\sum_{i=1}^{t}r_{i}^{-1}=1\) remains open. 
In 2026, Ding and Zhai \cite{ding2026} removed the integrality condition, obtaining the same conclusion for arbitrary positive real exponents.

In the polynomial setting, the first analogue of Romanoff's theorem was established by Shparlinski and Weingartner \cite{shparlinski2017} in 2017. 
Let \(q\) be a prime power and let \(\mathbb F_{q}\) denote the finite field with \(q\) elements. 
For \(f(x)\in\mathbb F_{q}[x]\), define its norm by \(|f|=q^{\deg f}\), and let \(\mathcal I_{q}\) denote the set of monic irreducible polynomials in \(\mathbb F_{q}[x]\); there are exactly \(q^{n}\) monic polynomials of degree \(n\). 
Set \(\mathcal I_{q}(n)=\{p\in\mathcal I_{q}:|p|=q^{n}\}\).

\begin{theorem}\label{thm:sw} \cite[Theorem 1.1]{shparlinski2017}
Let \(\gamma\) denote Euler's constant and set \(\delta=\deg g\). 
Uniformly for all polynomials \(g\in\mathbb F_{q}[x]\) with \(\delta\ge1\), and all integers \(n\ge1\), \(q\ge2\), we have
\[
\frac{1+\delta/n}{\delta}\ge\frac{R(n,g,q)}{q^{n}}>\frac{(1-2q^{-n/2})^{2}(1+\delta/n)^{-1}}{\delta+8\tfrac{q}{q-1}\left(1+e^{\gamma}\min\{5\sqrt{\delta/q},\ \log 6\delta/\log q\}\right)},
\]
where \(R(n,g,q)\) denotes the number of monic polynomials \(f\in\mathbb F_{q}[x]\) of degree \(n\) that admit a representation \(f=h+g^{k}\) with \(h\in\mathcal I_{q}(n)\) and \(k\ge0\).
\end{theorem}

For further results on Romanoff-type theorems in the polynomial setting, we refer the reader to  \cite{dai2026,dai2027,ding2021,ding2025}.

Motivated by the works above, in this paper we study the four corresponding questions over \(\mathbb F_{q}[x]\): Erd\H{o}s' multiplicity result (Theorem~\ref{thm:many}), non-unique representations (Theorem~\ref{thm:multi}), non-representable residue classes (Theorem~\ref{thm:nonrep}), and the positive-density theorem for real exponents (Theorem~\ref{thm:dens}). 
The precise statements are as follows.

Let \(g\in\mathbb F_{q}[x]\) be a given polynomial of degree \(\delta>0\). 
For a positive integer \(n\) and a monic polynomial \(f\) of degree \(n\), define the representation function
\[
R(f,n)=\left\{(h,k): f=h+g^{k},\ h\in\mathcal I_{q}(n),\ k\in\mathbb N,\ \delta k<n\right\},
\]
which counts the representations of \(f\) as the sum of a monic irreducible polynomial \(h\) of degree \(n\) and a power \(g^{k}\) of degree strictly less than \(n\); in contrast to Theorem~\ref{thm:sw}, the exponent \(k=0\) is excluded.

We record a simple consequence of the definition, used repeatedly below. 
Since \(\delta k<n\) with \(k\ge1\) forces \(n>\delta\), every monic irreducible divisor \(p\) of \(g\) satisfies \(\deg p\le\delta<n\). 
If \(f=h+g^{k}\) with \(h\in\mathcal I_{q}(n)\), \(k\ge1\) and \(\delta k<n\), then \(g^{k}\equiv0\pmod p\), so \(f\equiv h\pmod p\); as \(h\) is irreducible of degree \(n>\deg p\), it is not divisible by \(p\), whence \(p\nmid f\). 
Since \(p\mid g\) was arbitrary, \((f,g)=1\). 
Thus any polynomial admitting such a representation is coprime to \(g\).

Our first result is the function-field analogue of the multiplicity result of Erd\H{o}s recalled above. 
Since the size of a polynomial of degree \(n\) over \(\mathbb F_{q}\) is \(q^{n}\), the bound \(\gg\log n\) below is the function-field analogue of his result.

\begin{theorem}\label{thm:many}
There exists a sequence of monic polynomials \(f_{n}\in\mathbb F_{q}[x]\) with \(\deg f_{n}=n\) such that, for all sufficiently large \(n\),
\[
R(f_{n},n)\gg\log n,
\]
where the implied constant depends only on \(q\) and \(g\). 
In particular, \(R(f_{n},n)\to\infty\) as \(n\to\infty\).
\end{theorem}

Whereas Theorem~\ref{thm:sw} shows that a positive proportion of monic polynomials of degree \(n\) admit at least one representation, one may ask how many admit more than one. 
In the integer setting, Elsholtz and Schlage-Puchta \cite{elsholtz2018} proved that a positive proportion of integers admit exactly one representation as \(p+2^{k}\); the following theorem addresses the analogue in the opposite direction. 
As observed above, any polynomial admitting a representation is coprime to \(g\); for \(n\ge\delta\), the proportion of monic polynomials of degree \(n\) coprime to \(g\) is
\[
\rho_{g}=\prod_{p\mid g}\left(1-\frac1{|p|}\right)
=\frac{\varphi_{q}(g)}{|g|},
\]
where \(\varphi_q\) denotes the Euler function of \(\mathbb F_q[x]\), defined by \(\varphi_q(g)=|(\mathbb F_q[x]/(g))^\times|\), and the product runs over the monic irreducible divisors \(p\) of \(g\). 
Let \(\mathcal{S}_{n}\) be the set of monic polynomials \(f\) of degree \(n\) with \(R(f,n)>1\), and denote its cardinality by \(S(n)\).

\begin{theorem}\label{thm:multi}
If \(\delta\rho_{g}<1\), then \(S(n)\gg q^{n}\) for all sufficiently large \(n\), where the implied constant depends only on \(q\) and \(g\).
\end{theorem}

For positive real numbers \(r_{1},\dots,r_{t}\), let \(\mathcal{R}_{n}(g,q;r_{1},\dots,r_{t})\) denote the set of monic polynomials \(f\in\mathbb F_{q}[x]\) of degree \(n\) that admit a representation of the form
\[
f=h+g^{\lfloor k_{1}^{r_{1}}\rfloor}+\cdots+g^{\lfloor k_{t}^{r_{t}}\rfloor},
\]
where \(h\in\mathcal I_{q}(n)\) and \(k_{1},\dots,k_{t}\in\mathbb N\) satisfy \(\delta\lfloor k_{i}^{r_{i}}\rfloor<n\) for each \(i\), and set
\[
R_{n}(g,q;r_{1},\dots,r_{t})=|\mathcal{R}_{n}(g,q;r_{1},\dots,r_{t})|.
\]

The following theorem is the polynomial counterpart of the non-representability results recalled above. 
It holds for every exponent tuple, and, in contrast to the integer situation, where covering congruence systems are needed, a single residue class modulo an irreducible divisor of \(g-1\) suffices. 
In particular, whenever \(\sum_{i=1}^{t}r_i^{-1}\ge 1\), the non-representable polynomials already occupy a positive proportion of the monic polynomials of degree \(n\), so that both representable and non-representable polynomials have positive density.

\begin{theorem}\label{thm:nonrep}
Let \(r_{1},\dots,r_{t}>0\). 
Then there exists a monic irreducible polynomial \(p\in\mathbb F_{q}[x]\) such that for every \(n>\deg p\), if \(f\) is a monic polynomial of degree \(n\) satisfying \(f\equiv t\pmod p\), then \(f\notin \mathcal{R}_{n}(g,q;r_{1},\dots,r_{t})\).
\end{theorem}

Finally, combining the Romanoff-type machinery with the polynomial structure of \(\mathbb F_{q}[x]\), we prove the function-field analogue of the density theorem of Chen--Xu and Ding--Zhai described above.

\begin{theorem}\label{thm:dens}
Let \(r_{1},\dots,r_{t}\) be positive real numbers satisfying \(r_{1}^{-1}+\cdots+r_{t}^{-1}\ge1\). Then \(R_{n}(g,q;r_{1},\dots,r_{t})\gg q^{n}\) for all sufficiently large integers \(n\), where the implied constant depends only on \(q\), \(g\), and \(r_{1},\dots,r_{t}\).
\end{theorem}
\section{Notation and Preliminary Lemmas}
In this section we collect the auxiliary results needed in the proofs of the main theorems.

We first require two standard estimates on irreducible polynomials: the classical count of monic irreducibles of degree \(n\), and an upper bound for twin irreducible pairs \(h,\,h+f\in\mathcal I_q(n)\) with a fixed \(f\) of degree less than \(n\). 
Both are polynomial analogues of classical theorems over the integers; the latter is due to Pollack \cite{pollack2008}, in the form presented by Shparlinski and Weingartner \cite{shparlinski2017}. We record them as follows.

\begin{lemma}\label{lem:irred-count}\cite[Equation (2.1)]{shparlinski2017}
For \(n\ge1\), we have
\[
\frac{q^n}{n}-\frac{2q^{n/2}}{n}<|\mathcal I_q(n)|\le\frac{q^n}{n}.
\]
\end{lemma}

\begin{lemma}\label{lem:twin}\cite[Lemma 3.6]{shparlinski2017}
For any nonzero polynomial \(f(x)\in\mathbb F_q[x]\) and any integer
\(n>\deg f\),
\[
\left|\{h\in\mathcal I_q(n): h(x)+f(x)\in\mathcal I_q(n)\}\right|
\le\frac{8q^{n+1}}{n^2(q-1)}
\prod_{\substack{p\in\mathcal I_q\\p\mid f}}\left(1+\frac1{|p|}\right).
\]
\end{lemma}

Next we need a count of irreducible polynomials in a prescribed residue class. 
Let \(M\in\mathbb F_q[x]\) be a fixed polynomial of degree \(N\), let \(m\ge1\), and let \(h\in\mathbb F_q[x]\) satisfy \((h,M)=1\). 
Denote by \(\pi_m(h)\) the number of monic irreducible polynomials \(p\in\mathcal I_q(m)\) with \(p\equiv h\pmod M\). 
The asymptotic
\[
\pi_m(h)=\frac{q^m}{m\varphi_q(M)}+O\!\left(\frac{q^{m/2}}{m}\right)
\]
is classical (see, e.g., \cite[Theorem 4.8]{rosen2002}). 
In the proof of Theorem~\ref{thm:many} we require a version of this estimate that is explicit and uniform in the modulus; the following bound of Wan \cite[Theorem 5.1]{wan1997} serves this purpose.

\begin{lemma}\label{lem:wan}
Let \(M\in\mathbb F_q[x]\) have degree \(N\), let \(m\ge1\), and let
\(h\in\mathbb F_q[x]\) satisfy \((h,M)=1\). Then
\[
\left|\pi_m(h)-\frac{q^m}{m\varphi_q(M)}\right|
\le\frac{(N+1)q^{m/2}}{m}.
\]
\end{lemma}

For real \(r>1\), positive integer \(m\), \(\ell\in\mathbb Z\), and \(K>0\), we consider the congruence
\[
\lfloor k^r\rfloor\equiv\ell\pmod m
\]
and denote by \(N(K,\ell,m)\) the number of its solutions with \(1\le k\le K\). Chen--Xu \cite[Theorem 3.1]{chen2024} obtained an upper bound for \(N(K,\ell,m)\) in the case of integer exponents \(r\), while for non-integer \(r\), Ding--Zhai \cite{ding2026} employed the method of exponent pairs for the specific value of \(K\) arising in their application. 
We record the integer case as (i); the non-integer case (ii), which holds for arbitrary \(K\), is the form required in the proof of Theorem~\ref{thm:dens}, and we prove it in full.

\begin{lemma}\label{lem:cong}
Let \(r>1\), \(m\) be a positive integer, \(\ell\in\mathbb Z\), and \(K>0\), and let \(N(K,\ell,m)\) be as above.
\begin{enumerate}[label=(\roman*)]
\item \cite[Theorem 3.1]{chen2024} If \(r\in\mathbb Z\) and \(r\ge2\), then for every \(\varepsilon>0\) there exists a constant
\(c=c(r,\varepsilon)>0\) such that
\[
N(K,\ell,m)\le c\left(Km^{-1/r}+m^{\varepsilon}\right),
\]
and if \(m\ge K^r\), then \(N(K,\ell,m)\le 1\).
\item If \(r\notin\mathbb Z\), let \(b=\lfloor r\rfloor+1\) and \(B=2^b\). Then there exists a constant
\(c=c(r)>0\) such that
\[
N(K,\ell,m)\le c\left(\frac{K}{m}+K^{1-\eta}m^{-\alpha}+m^{1/r}\right),
\quad
\eta=\frac{b+2-r}{4B-1},\quad \alpha=\frac{1}{4B-1}.
\]
In particular, \(\eta=\eta(r)>0\) and \(\alpha=\alpha(r)>0\) depend only on \(r\).
\end{enumerate}
\end{lemma}
\begin{proof}
(ii) All implied constants below depend only on \(r\), and \(\psi(t)=\{t\}-1/2\) denotes the sawtooth function. 
Put \(C_0=\max\{3,m^{1/r}\}\). 
We may assume \(K>C_0\); otherwise, since \(m^{1/r}\ge1\),
\[
N(K,\ell,m)\le K\le C_0\ll m^{1/r}.
\]

An integer \(j\ge0\) satisfies
\[
\frac{k^r-\ell-1}{m}<j\le\frac{k^r-\ell}{m}
\]
if and only if
\[
\left\lfloor\frac{k^r-\ell}{m}\right\rfloor
-\left\lfloor\frac{k^r-\ell-1}{m}\right\rfloor=1.
\]
Summing over \(k\le K\) and inserting
\(\lfloor t\rfloor=t-\tfrac12-\psi(t)\), it follows that
\begin{equation}\label{2.1}
N(K,\ell,m)\ll\frac{K}{m}+\max_{|\theta|\le1}
\left|\sum_{k\le K}\psi\left(\frac{k^r}{m}+\theta\right)\right|.
\end{equation}
Write the \(\psi\)-sum as \(S_1+S_2\), where \(S_1\) is over
\(k\le C_0\) and \(S_2\) is over \(C_0<k\le K\). Trivially,
\[
S_1\ll C_0\ll m^{1/r}.
\]
For \(S_2\), by the van der Corput bound \cite[Lemma 4.3]{graham1991}
with \(Y=1/m\), for any exponent pair \((\kappa,\lambda)\), any real
\(K_0\ge3\) and any \(K'\) with \(K_0\le K'\le2K_0\),
\begin{equation}\label{2.2}
\sum_{K_0\le k<K'}\psi\left(\frac{k^r}{m}+\theta\right)
\ll mK_0^{1-r}+m^{-\frac{\kappa}{1+\kappa}}K_0^{\frac{\lambda+\kappa r}{1+\kappa}}.
\end{equation}
By \cite[p.~60]{graham1991},
\[
(\kappa,\lambda)=\left(\frac{1}{4B-2},\,1-\frac{b+1}{4B-2}\right)
\]
is an exponent pair, for which
\[
\frac{\kappa}{1+\kappa}=\frac{1}{4B-1}=\alpha,\qquad
\frac{\lambda+\kappa r}{1+\kappa}=1-\frac{b+2-r}{4B-1}=1-\eta,
\]
and \(0<\eta<1\) since \(r<b+2\).

Let \(K_0\) run over the dyadic values \(K_j=2^jC_0\)
(\(j=0,1,\dots,J\)) with \(K_J\le K<2K_J\), and apply
(\ref{2.2}) to the intervals \([K_j,2K_j)\) for \(j<J\) and to the final
interval \([K_J,K)\) for \(j=J\), which is admissible since
\(K_J\ge C_0\ge3\) and \(K\le2K_J\). The contribution of the first term
on the right-hand side of (\ref{2.2}) is
\[
\sum_{j=0}^{J}mK_j^{1-r}
=mC_0^{1-r}\sum_{j=0}^{J}2^{-j(r-1)}
\ll_r mC_0^{1-r}\le m^{1/r},
\]
since \(r>1\) and \(C_0\ge m^{1/r}\); that of the second is
\[
\sum_{j=0}^{J}m^{-\alpha}K_j^{1-\eta}
=m^{-\alpha}C_0^{1-\eta}\sum_{j=0}^{J}2^{j(1-\eta)}
\ll_r m^{-\alpha}C_0^{1-\eta}\left(\frac{K}{C_0}\right)^{1-\eta}
=m^{-\alpha}K^{1-\eta},
\]
since \(1-\eta>0\). Combining these with a splitting argument yields
\[
S_2=\sum_{C_0<k\le K}\psi\left(\frac{k^r}{m}+\theta\right)
\ll_r m^{1/r}+m^{-\alpha}K^{1-\eta},
\]
which together with the estimate \(S_1\ll m^{1/r}\) and
(\ref{2.1}) proves the lemma.
\end{proof}

A key step in Romanoff-type problems is the convergence of the series
\[
\sum_{\substack{m\ge1\\(m,a)=1}}\frac{1}{m\operatorname{ord}_m(a)}<\infty,
\]
where \(\operatorname{ord}_m(a)\) denotes the multiplicative order of
\(a\) modulo \(m\). 
Romanoff \cite{romanoff1934} proved this convergence, which drives his density argument. 
In 1935, Landau \cite{landau1935} obtained the explicit bound
\[
\sum_{\substack{m\ge1\\(m,a)=1}}\frac{1}{m\operatorname{ord}_m(a)}
\ll(\log\log a)^2,
\]
and Erd\H{o}s--Tur\'an \cite{erdos1935} proved the sharper estimate, for every \(\varepsilon>0\),
\[
\sum_{\substack{m\ge1\\(m,a)=1}}\frac{1}{m\operatorname{ord}_m(a)^{\varepsilon}}
\ll_{\varepsilon}\log\log a,
\]
which has since been the standard tool in Romanoff-type problems. 
Murty, Rosen and Silverman \cite{murty1996} made these bounds explicit in 1996 and extended them to number fields and abelian varieties, and Kuan \cite{kuan2015} established the corresponding function-field and Drinfeld module analogues in 2015.

The following lemma is the case of the rational function field
\(\mathbb F_q(x)\) and the rank-one subgroup \(\Gamma=\langle g\rangle\)
of Kuan's theorem \cite[Theorem 2]{kuan2015}, restricted to squarefree moduli and written in polynomial notation. 
For completeness and for later reference, we include a proof; the argument will reappear in modified form in the proofs of the main theorems.

\begin{lemma}\label{lem:conv}
Let \(g\in\mathbb F_q[x]\) be a fixed nonconstant polynomial. For any
real number \(\varepsilon>0\), the series
\[
\sum_{\substack{f\ \text{monic}\\(f,g)=1}}
\frac{\mu^2(f)}{|f|\cdot\operatorname{ord}_f(g)^{\varepsilon}}
\]
converges, where the sum is over monic polynomials \(f\) and \(\operatorname{ord}_f(g)\) denotes the multiplicative order of \(g\) modulo \(f\), with the convention \(\operatorname{ord}_1(g)=1\).
\end{lemma}

\begin{proof}
For \(c\in\mathbb F_q^\times\), multiplication by \(c\) preserves \(|f|\), \((f,g)\), \(\operatorname{ord}_f(g)\), and \(\mu^2(f)\), so convergence over all nonzero \(f\) follows from the monic case up to the factor \(q-1\).

We first prove that for any nonzero \(h\in\mathbb F_q[x]\) with
\(|h|>e\),
\begin{equation}\label{2.3}
\sum_{\substack{f\ \text{monic}\\f\mid h}}\frac1{|f|}
\ll\log\log|h|. 
\end{equation}
Indeed, by unique factorization,
\[
\sum_{f\mid h}\frac1{|f|}
\le\prod_{p\mid h}\left(1-\frac1{|p|}\right)^{-1},
\]
where the product runs over monic irreducible divisors of \(h\). 
If \(u\) of these satisfy \(|p|>\log|h|\), then \(u\ll\log|h|/\log\log|h|\), and
\[
\prod_{\substack{p\mid h\\|p|>\log|h|}}\left(1-\frac1{|p|}\right)^{-1}
\le\left(1-\frac1{\log|h|}\right)^{-u}
\le\exp\left(\frac{u}{\log|h|-1}\right)=O(1).
\]
For the remaining primes, the function-field Mertens theorem \cite[Theorem 3]{rosen1999} yields
\[
\prod_{\substack{p\mid h\\|p|\le\log|h|}}\left(1-\frac1{|p|}\right)^{-1}
\le\prod_{|p|\le\log|h|}\left(1-\frac1{|p|}\right)^{-1}
\ll\log\log|h|.
\]
These bounds are uniform for all sufficiently large \(|h|\); the finitely many smaller \(h\) are absorbed into the implied constant, so (\ref{2.3}) holds.

For \(m\ge1\), set
\[
d(m)=\sum_{\substack{f\ \text{monic}\\(f,g)=1\\
\operatorname{ord}_f(g)=m}}\frac{\mu^2(f)}{|f|}.
\]
We claim that
\begin{equation}\label{2.4}
\sum_{m\le T}d(m)\ll\log T 
\end{equation}
for all \(T\ge2\). 
Let
\[
A_T=\prod_{k=1}^{T}(g^k-1).
\]
If \(\operatorname{ord}_f(g)=m\le T\), then \(f\mid g^m-1\), hence \(f\mid A_T\). 
Consequently,
\[
\sum_{m\le T}d(m)\le\sum_{\substack{f\ \text{monic}\\f\mid A_T}}\frac1{|f|}.
\]
Since \(g\) is nonconstant, \(\deg(g^k-1)=k\deg g\), so
\[
\deg A_T=\sum_{k=1}^{T}k\deg g=\frac{T(T+1)}2\deg g\le T^2\deg g.
\]
As \(\deg A_T\ge3\), (\ref{2.3}) applies with \(h=A_T\) and gives
\[
\sum_{m\le T}d(m)\ll\log\log|A_T|\ll\log T,
\]
which proves (\ref{2.4}).

Finally, let
\[
D(t)=\sum_{m\le t}d(m).
\]
By (\ref{2.4}), \(D(t)\ll\log t\) for \(t\ge2\), while \(D(t)\le D(2)=O(1)\) for \(1\le t<2\). 
By Abel summation,
\[
\sum_{m\le N}\frac{d(m)}{m^{\varepsilon}}
=\frac{D(N)}{N^{\varepsilon}}
+\varepsilon\int_1^N\frac{D(t)}{t^{1+\varepsilon}}\,dt.
\]
The boundary term tends to \(0\), since \(D(N)N^{-\varepsilon}\ll(\log N)N^{-\varepsilon}\to0\). The integral is bounded by
\[
\varepsilon\int_1^2\frac{dt}{t^{1+\varepsilon}}
+\varepsilon\int_2^{\infty}\frac{\log t}{t^{1+\varepsilon}}\,dt<\infty.
\]
Since \(d(m)\ge0\), the partial sums are increasing, so the series over monic polynomials converges. 
\end{proof}
\section{Proofs of the Theorems}
Throughout this section, all irreducible polynomials are monic, so that a product over \(p\mid H\) runs over the monic irreducible divisors of \(H\). 
All sums and counts over polynomials \(f\) are over monic polynomials of degree \(n\) unless otherwise indicated.

\begin{proof}[\textbf{Proof of Theorem~\ref{thm:many}}]
Set
\[
m=\left\lfloor\frac12\log_q n\right\rfloor,\qquad
M=\prod_{\substack{\deg p\le m\\p\nmid g}}p,
\]
the product being over monic irreducible polynomials \(p\), and write
\(N=\deg M\). By Lemma~\ref{lem:irred-count},
\[
N=\sum_{\substack{\deg p\le m\\p\nmid g}}\deg p
=\sum_{k\le m}k\,|\mathcal I_q(k)|+O(\delta),
\]
the error \(O(\delta)\) accounting for the monic irreducible divisors of
\(g\), whose total degree is at most \(\delta=\deg g\). Applying
Lemma~\ref{lem:irred-count} again,
\[
\sum_{k\le m}k\,|\mathcal I_q(k)|
=\sum_{k\le m}q^{k}+O\!\left(\sum_{k\le m}q^{k/2}\right)
=\frac{q}{q-1}q^{m}+O(q^{m/2}),
\]
so that
\[
N=\frac{q}{q-1}q^{m}+O(q^{m/2})\asymp q^{m}\asymp\sqrt n.
\]
By the function-field Mertens theorem,
\begin{equation}\label{3.1}
\frac{|M|}{\varphi_q(M)}
=\prod_{\substack{\deg p\le m\\p\nmid g}}\left(1-\frac1{|p|}\right)^{-1}
\asymp m. 
\end{equation}
All implied constants from now on depend only on \(q\) and \(\delta\).

Consider the set
\[
S=\left\{(h,k): h\in\mathcal I_q(n),\ k\in\mathbb N,\ \delta k<n,\
h+g^k\equiv0\pmod M\right\}.
\]
There are \(n/\delta+O(1)\) admissible values of \(k\). For each such
\(k\), since \((g,M)=1\) by construction, the residue class
\(-g^k\bmod M\) is coprime to \(M\), and Lemma~\ref{lem:wan} gives
\[
\left|\{h\in\mathcal I_q(n): h\equiv-g^k\pmod M\}\right|
=\frac{q^n}{n\varphi_q(M)}+O\!\left(\frac{Nq^{n/2}}n\right).
\]
Summing over the admissible \(k\), we obtain
\begin{equation}\label{3.2}
S=\frac{q^n}{\delta\varphi_q(M)}
+O\!\left(\frac{q^n}{n\varphi_q(M)}\right)+O\!\left(Nq^{n/2}\right). 
\end{equation}
Each pair \((h,k)\) counted by \(S\) determines the monic polynomial \(f=h+g^k\), which has degree \(n\) because \(\deg h=n\) while \(\deg g^k=\delta k<n\), and which satisfies \(M\mid f\). 
Grouping the pairs according to \(f\), we get
\[
S=\sum_{M\mid f}R(f,n).
\]
There are \(q^{n-N}\) monic polynomials \(f\) of degree \(n\) divisible by \(M\), provided \(n\ge N\); this holds for all sufficiently large \(n\), since \(N\asymp\sqrt n\). 
Averaging, some monic \(f_n\) of degree \(n\) satisfies
\begin{equation}\label{3.3}
R(f_n,n)\ge\frac{S}{q^{n-N}}
=\frac1\delta\frac{|M|}{\varphi_q(M)}
+O\!\left(\frac1n\frac{|M|}{\varphi_q(M)}\right)+O\!\left(Nq^{N-n/2}\right).
\end{equation}
By (\ref{3.1}), the main term is \(\asymp m/\delta\). The first error term is \(O(m/n)=o(1)\). 
Since \(N\asymp\sqrt n\), the second error term is
\(O\!\left(q^{O(\sqrt n)-n/2}\right)=o(1)\). 
Therefore
\[
R(f_n,n)\gg m/\delta\gg\log n,
\]
which proves the theorem.
\end{proof}

\begin{proof}[\textbf{Proof of Theorem~\ref{thm:multi}}]
Consider the first moment
\[
T_n=\sum_{\deg f=n}R(f,n)
=\left|\{(h,k): h\in\mathcal I_q(n),\ k\in\mathbb N,\ \delta k<n\}\right|.
\]
There are \(|\mathcal I_q(n)|\) choices for \(h\) and \(n/\delta+O(1)\) choices for \(k\), whence, by Lemma~\ref{lem:irred-count},
\[
T_n=\left(\frac{q^n}{n}+O\!\left(\frac{q^{n/2}}n\right)\right)
\left(\frac n\delta+O(1)\right)
=\frac{q^n}{\delta}+O\!\left(\frac{q^n}{n}\right),
\]
the last step following from \(q^{n/2}\ll q^n/n\) for sufficiently large \(n\).

Splitting \(\sum_f(R(f,n)-1)\) according to whether \(R(f,n)=0\), \(R(f,n)=1\), or \(R(f,n)>1\), and noting that the terms with
\(R(f,n)=1\) contribute nothing, we get
\[
\sum_{f\in \mathcal{S}_n}\left(R(f,n)-1\right)
=T_n-q^n+\left|\{f:R(f,n)=0\}\right|.
\]
Every \(f\) with \((f,g)>1\) satisfies \(R(f,n)=0\), and for \(n\ge\delta\) exactly \(\rho_gq^n\) monic polynomials of degree \(n\) are coprime to \(g\); hence
\[
\left|\{f:R(f,n)=0\}\right|\ge(1-\rho_g)q^n,
\]
and therefore
\begin{equation}\label{3.4}
\sum_{f\in \mathcal{S}_n}\left(R(f,n)-1\right)
\ge T_n-\rho_gq^n
=\left(\frac1\delta-\rho_g\right)q^n+O\!\left(\frac{q^n}{n}\right).
\end{equation}
Since \(\delta\rho_g<1\), the coefficient \(1/\delta-\rho_g\) is positive.

It remains to bound the second moment. Shparlinski--Weingartner \cite[Eq.(4.3)]{shparlinski2017} proved
\begin{equation}\label{3.5}
\sum_{\deg f=n}R(f,n)^2\ll_{q,g}q^n; 
\end{equation}
indeed, their estimate concerns the representation function \(C(f,n)\), which counts in addition the exponent \(k=0\), and since \(R(f,n)\le C(f,n)\) pointwise, (\ref{3.5}) follows.

Combining (\ref{3.4}) and (\ref{3.5}) with the trivial bound \(R(f,n)-1\le R(f,n)\) and the Cauchy--Schwarz inequality, we obtain
\begin{align*}
\left(\frac1\delta-\rho_g\right)q^n+O\!\left(\frac{q^n}{n}\right)
&\le \sum_{f\in\mathcal S_n}\left(R(f,n)-1\right)
\le \sum_{f\in\mathcal S_n}R(f,n) \\
&\le \left(\sum_f R(f,n)^2\right)^{1/2}S(n)^{1/2}
\ll_{q,g} q^{n/2}S(n)^{1/2}.
\end{align*}
Since \(1/\delta-\rho_g>0\) and \(q^n/n=o(q^n)\), it follows that \(S(n)\gg q^n\) for all sufficiently large \(n\), where the implied constant depends only on \(q\) and \(g\).
\end{proof}

\begin{remark}
The argument above used only the trivial lower bound
\[
\left|\{f:R(f,n)=0\}\right|\ge(1-\rho_g)q^n.
\] 
Theorem~\ref{thm:nonrep} below provides a stronger one and thereby relaxes the hypothesis of Theorem~\ref{thm:multi}. Indeed, let \(p\) be a monic irreducible divisor of \(g-1\); then \(p\nmid g\). 
For \(n\ge\delta+\deg p\), the Chinese remainder theorem shows that exactly
\[
\varphi_q(g)q^{n-\delta-\deg p}=\frac{\rho_g}{|p|}q^n
\]
monic polynomials \(f\) of degree \(n\) satisfy \(f\equiv1\pmod p\) and \((f,g)=1\). 
By the case \(t=1\) of Theorem~\ref{thm:nonrep}, each such
\(f\) has \(R(f,n)=0\), and these \(f\) are disjoint from the \((1-\rho_g)q^n\) polynomials with \((f,g)>1\). 
Hence
\[
\left|\{f:R(f,n)=0\}\right|\ge\left(1-\rho_g+\frac{\rho_g}{|p|}\right)q^n,
\]
and (\ref{3.4}) improves to
\[
\sum_{f\in \mathcal{S}_n}\left(R(f,n)-1\right)
\ge\left(\frac1\delta-\rho_g+\frac{\rho_g}{|p|}\right)q^n
+O\!\left(\frac{q^n}{n}\right).
\]
Consequently the hypothesis \(\delta\rho_g<1\) may be relaxed to \(\delta\rho_g(1-1/|p|)<1\). 
We shall not use this refinement.
\end{remark}

\begin{proof}[\textbf{Proof of Theorem~\ref{thm:nonrep}}]
Since \(\delta\ge1\), the polynomial \(g-1\) is nonconstant of degree
\(\delta\). By unique factorization in \(\mathbb F_q[x]\), we may choose
a monic irreducible divisor \(p\) of \(g-1\). Then \(g\equiv1\pmod p\),
and in particular \(p\nmid g\).

For every nonnegative integer \(k\), the congruence
\(g\equiv1\pmod p\) implies \(g^k\equiv1\pmod p\). Since each
\(\lfloor k_i^{r_i}\rfloor\) is a nonnegative integer, every summand
\(g^{\lfloor k_i^{r_i}\rfloor}\) is congruent to \(1\) modulo \(p\).
Hence every admissible power sum
\[
a=\sum_{i=1}^{t}g^{\lfloor k_i^{r_i}\rfloor}
\]
satisfies \(a\equiv t\pmod p\), because there are exactly \(t\) summands.

Now let \(f\) be a monic polynomial of degree \(n>\deg p\) with
\(f\equiv t\pmod p\). For any admissible power sum \(a\), we have
\(a\equiv t\equiv f\pmod p\), so \(p\mid f-a\). Since \(\deg a<n\), the
leading term of \(f-a\) coincides with that of \(f\), and therefore
\(\deg(f-a)=n>\deg p\). Thus \(p\) is a nonconstant proper divisor of
\(f-a\), so \(f-a\) is reducible and cannot be an irreducible polynomial
of degree \(n\). Therefore no such \(f\) admits a representation
\(f=h+a\) with \(h\in\mathcal I_q(n)\). This proves the theorem.
\end{proof}

\begin{proof}[\textbf{Proof of Theorem~\ref{thm:dens}}]
All implied constants in this proof depend only on \(q\), \(g\) and \(r_1,\dots,r_t\). Suppose first that \(r_{i_0}\le1\) for some \(i_0\). Since the steps of the sequence \(\{k^{r_{i_0}}\}_{k\ge1}\) do not exceed \(1\), the sequence \(\{\lfloor k^{r_{i_0}}\rfloor:k\in\mathbb N\}\) contains every positive integer; hence \(\{g^{\lfloor k^{r_{i_0}}\rfloor}:k\in\mathbb N\}\) contains all powers \(g^m\) with \(m\ge1\). 
Taking \(k_i=1\) for \(i\ne i_0\), each summand for \(i\ne i_0\) equals \(g\), and
\[
h+g^{m}+(t-1)g,\qquad h\in\mathcal I_q(n),\quad \delta m<n,
\]
lies in \(\mathcal{R}_n(g,q;r_1,\dots,r_t)\). 
In other words, \(\mathcal{R}_n(g,q;r_1,\dots,r_t)\) contains a translate of the set counted by \(R(n,g,q)\) in Theorem~\ref{thm:sw}. 
Restricting the exponent in Theorem~\ref{thm:sw} to \(m\ge1\) with \(\delta m<n\) loses only the value \(m=0\), hence only \(O(q^n/n)\) polynomials, so Theorem~\ref{thm:sw} still yields a positive proportion, and the assertion follows. 
Henceforth we assume that \(r_i>1\) for all \(1\le i\le t\).

Since \(\sum_{i=1}^{t}r_i^{-1}\ge1\), there exists a smallest integer \(s\ge2\) such that
\[
\sum_{i=1}^{s-1}\frac1{r_i}<1\le\sum_{i=1}^{s}\frac1{r_i},
\]
and therefore a real number \(\lambda\ge1\) satisfying
\begin{equation}\label{3.6}
\sum_{i=1}^{s-1}\frac1{r_i}+\frac1{\lambda r_s}=1. 
\end{equation}
Put \(r'_i=r_i\) for \(1\le i\le s-1\) and \(r'_s=\lambda r_s\), so that
\(\sum_{i=1}^{s}1/r'_i=1\). 
Define
\[
\mathcal{A}=\left\{\sum_{i=1}^{s-1}g^{\lfloor k_i^{r_i}\rfloor}
+g^{\lfloor\lfloor k_s^{\lambda}\rfloor^{r_s}\rfloor}:
k_i\in\mathbb N,\ \delta\lfloor k_i^{r_i}\rfloor<n\ (1\le i\le s-1),\
\delta\lfloor\lfloor k_s^{\lambda}\rfloor^{r_s}\rfloor<n\right\}.
\]
Every element of \(\mathcal{A}\) has degree \(<n\). 
Setting \(k_{s+1}=\dots=k_t=1\), we obtain
\[
\{a+(t-s)g: a\in\mathcal{A}\}\subseteq
\left\{\sum_{i=1}^{t}g^{\lfloor k_i^{r_i}\rfloor}:k_i\in\mathbb N,
\delta\lfloor k_i^{r_i}\rfloor<n\right\},
\]
so that \(\mathcal{R}_n(g,q;r_1,\dots,r_t)\) contains the translate of \(\mathcal I_q(n)+\mathcal A\) by the fixed polynomial \((t-s)g\).
It therefore suffices to prove that \(\mathcal I_q(n)+\mathcal{A}\) has positive proportion among the monic polynomials of degree \(n\).

We first prove that \(|\mathcal{A}|\asymp n\). If \(k_1,\dots,k_s,l_1,\dots,l_s\) are positive integers such that
\[
\lfloor k_1^{r_1}\rfloor,\dots,\lfloor k_{s-1}^{r_{s-1}}\rfloor,
\left\lfloor\lfloor k_s^{\lambda}\rfloor^{r_s}\right\rfloor
\]
are distinct, and
\[
\lfloor l_1^{r_1}\rfloor,\dots,\lfloor l_{s-1}^{r_{s-1}}\rfloor,
\left\lfloor\lfloor l_s^{\lambda}\rfloor^{r_s}\right\rfloor
\]
are also distinct, then, since each term \(g^{m}\) has degree exactly \(\delta m\) and the exponents within each sum are distinct, no cancellation occurs and the highest-degree term of each sum is unique.
By induction on the highest degree, the two sums
\[
\sum_{i=1}^{s-1}g^{\lfloor k_i^{r_i}\rfloor}
+g^{\lfloor\lfloor k_s^{\lambda}\rfloor^{r_s}\rfloor}
=
\sum_{i=1}^{s-1}g^{\lfloor l_i^{r_i}\rfloor}
+g^{\lfloor\lfloor l_s^{\lambda}\rfloor^{r_s}\rfloor}
\]
are equal if and only if
\[
\left(\lfloor k_1^{r_1}\rfloor,\dots,\lfloor k_{s-1}^{r_{s-1}}\rfloor,
\left\lfloor\lfloor k_s^{\lambda}\rfloor^{r_s}\right\rfloor\right)
\]
is a permutation of
\[
\left(\lfloor l_1^{r_1}\rfloor,\dots,\lfloor l_{s-1}^{r_{s-1}}\rfloor,
\left\lfloor\lfloor l_s^{\lambda}\rfloor^{r_s}\right\rfloor\right).
\]
It follows that \(|\mathcal{A}|\) is at least the number of \(s\)-tuples \((m_1,\dots,m_s)\) of positive integers with
\[
m_1<m_2<\cdots<m_s\le\frac n\delta
\]
for which there exist \(k_1,\dots,k_s\in\mathbb N\) such that
\[
\left\{\lfloor k_1^{r_1}\rfloor,\dots,\lfloor k_{s-1}^{r_{s-1}}\rfloor,
\left\lfloor\lfloor k_s^{\lambda}\rfloor^{r_s}\right\rfloor\right\}
=\{m_1,\dots,m_s\}.
\]
Since, for each fixed \(i\), a given value \(m_j\) corresponds to at most one \(k_i\), the number of such \(s\)-tuples \((k_1,\dots,k_s)\) is at most \(s!\) times the number of strictly increasing \(s\)-tuples \((m_1,\dots,m_s)\). 
Hence \(s!|\mathcal{A}|\) is at least the number of \(s\)-tuples \((k_1,\dots,k_s)\in\mathbb N^{s}\) for which
\[
\lfloor k_1^{r_1}\rfloor,\dots,\lfloor k_{s-1}^{r_{s-1}}\rfloor,
\left\lfloor\lfloor k_s^{\lambda}\rfloor^{r_s}\right\rfloor
\]
are distinct and do not exceed \(n/\delta\). The number of \(s\)-tuples \((k_1,\dots,k_s)\in\mathbb N^{s}\) satisfying \(\delta\lfloor k_i^{r_i}\rfloor\le n\) for \(1\le i\le s-1\) and \(\delta\left\lfloor\lfloor k_s^{\lambda}\rfloor^{r_s}\right\rfloor\le n\)
is
\[
\gg\left(\frac n\delta\right)^{\frac1{r_1}+\cdots+\frac1{r_{s-1}}
+\frac1{\lambda r_s}}
=\frac n\delta\gg n.
\]
The number of such \(s\)-tuples for which at least two of the exponents are equal is
\[
\ll\left(\frac n\delta\right)^{\frac1{r_1}+\cdots+\frac1{r_{s-1}}
+\frac1{\lambda r_s}-\frac1{\max\{r_1,\dots,r_{s-1},\lambda r_s\}}}
\ll n^{1-\frac1{\max_{1\le i\le s}r'_i}}.
\]
Since the exponent is strictly smaller than \(1\), this is \(o(n)\).
Hence \(|\mathcal{A}|\gg n\).

On the other hand, each element \(a\in\mathcal{A}\) corresponds to at least one \(s\)-tuple \((k_1,\dots, k_s)\) satisfying the degree constraints. 
Thus
\[
|\mathcal{A}|\le\left|\left\{(k_1,\dots,k_s)\in\mathbb N^{s}:
\delta\lfloor k_i^{r_i}\rfloor\le n\ (1\le i\le s-1),\
\delta\left\lfloor\lfloor k_s^{\lambda}\rfloor^{r_s}\right\rfloor\le n
\right\}\right|.
\]
For each \(1\le i\le s-1\), the number of \(k_i\) with
\(\delta\lfloor k_i^{r_i}\rfloor\le n\) is \(\ll n^{1/r_i}\), and for \(i=s\) the number is \(\ll n^{1/(\lambda r_s)}\). 
Hence
\[
|\mathcal{A}|\ll n^{\frac1{r_1}+\cdots+\frac1{r_{s-1}}+\frac1{\lambda r_s}}=n.
\]
Combining the upper and lower bounds, we obtain \(|\mathcal{A}|\asymp n\).

For each monic polynomial \(f\) of degree \(n\), let \(T(f)\) denote the number of representations
\[
f=h+a,\qquad h\in\mathcal I_q(n),\ a\in \mathcal{A}.
\]
All sums over \(f\) below run over the monic polynomials of degree \(n\). 
Since \(|\mathcal I_q(n)|\asymp q^n/n\) by Lemma~\ref{lem:irred-count} and \(|\mathcal{A}|\asymp n\), the first moment gives
\begin{equation}\label{3.7}
\sum_fT(f)=|\mathcal I_q(n)|\cdot|\mathcal{A}|\gg q^n. 
\end{equation}
Moreover, if \(T(f)\ge1\), then \(f=h+a\) for some \(h\in\mathcal I_q(n)\) and \(a\in\mathcal{A}\), whence \(h+a+(t-s)g\in \mathcal{R}_n(g,q;r_1,\dots,r_t)\); distinct \(f\) give distinct such polynomials, and therefore
\[
|\{f:T(f)\ge1\}|\le R_{n}(g,q;r_{1},\dots,r_{t}).
\]
By the Cauchy--Schwarz inequality,
\begin{equation}\label{3.8}
\left(\sum_fT(f)\right)^2
\le\left(\sum_fT(f)^2\right)\sum_{T(f)\ge1}1
\le\left(\sum_fT(f)^2\right)R_{n}(g,q;r_{1},\dots,r_{t}). 
\end{equation}
Combining (\ref{3.7}) and (\ref{3.8}), we obtain
\[
R_{n}(g,q;r_{1},\dots,r_{t})\ge\frac{\left(\sum_fT(f)\right)^2}{\sum_fT(f)^2}
\gg\frac{q^{2n}}{\sum_fT(f)^2}.
\]
Therefore, to prove the theorem, it suffices to show that
\begin{equation}\label{3.9}
\sum_fT(f)^2\ll q^n.
\end{equation}

Expanding the square and exchanging the order of summation,
\[
\sum_fT(f)^2
=\sum_{\substack{h_1-h_2=a_2-a_1\\h_i\in\mathcal I_q(n),\ a_i\in\mathcal{A}}}1 .
\]
Write \(\Delta=h_1-h_2=a_2-a_1\), and split the sum according to \(\Delta=0\) or \(\Delta\ne0\). 
The contribution of \(\Delta=0\) is \(|\mathcal I_q(n)|\cdot|\mathcal{A}|\ll q^n\), which is acceptable. 
For \(\Delta\ne0\), put
\[
\mathcal{J}=\sum_{\substack{\Delta\ne0\\\deg\Delta<n}}
\left|\{h\in\mathcal I_q(n):h+\Delta\in\mathcal I_q(n)\}\right|
\cdot\left|\{(a_1,a_2)\in\mathcal{A}^2:a_2-a_1=\Delta\}\right|.
\]
By Lemma~\ref{lem:twin},
\[
\left|\{h\in\mathcal I_q(n):h+\Delta\in\mathcal I_q(n)\}\right|
\ll\frac{q^n}{n^2}\prod_{p\mid\Delta}\left(1+\frac1{|p|}\right),
\]
and hence
\[
\mathcal{J}\ll\frac{q^n}{n^2}
\sum_{\substack{a_1,a_2\in\mathcal{A}\\a_1\ne a_2}}
\prod_{\substack{p\mid a_2-a_1\\p\nmid g}}\left(1+\frac1{|p|}\right),
\]
the irreducible divisors of \(g\) having been absorbed into the implied constant.

We next remove the irreducible divisors of large norm. 
Since \(\deg(a_2-a_1)<n\), the number of irreducible divisors \(p\mid a_2-a_1\)
with \(|p|>n\) is \(O(n/\log n)\), and for each of them
\(1+1/|p|<1+1/n\). Hence
\[
\prod_{\substack{p\mid a_2-a_1\\|p|>n}}\left(1+\frac1{|p|}\right)
<\left(1+\frac1n\right)^{O(n/\log n)}\ll1,
\]
so only the irreducible divisors of norm at most \(n\) contribute. 
Using the Euler product identity for squarefree divisors,
\[
\prod_{\substack{p\mid a_2-a_1\\|p|\le n,\ p\nmid g}}\left(1+\frac1{|p|}\right)
=\sum_{\substack{d\mid a_2-a_1\\(d,g)=1\\P^+(d)\le n}}\frac{\mu^2(d)}{|d|},
\]
where \(P^+(d)\) denotes the maximal norm of a monic irreducible divisor of \(d\), with the convention \(P^+(1)=1\). 
All sums over \(d\) in the remainder of the proof are over monic polynomials.
Substituting into the estimate for \(\mathcal{J}\), we obtain
\begin{equation}\label{3.10}
\mathcal{J}\ll\frac{q^n}{n^2}
\sum_{\substack{(d,g)=1\\P^+(d)\le n}}\frac{\mu^2(d)}{|d|}S(d), 
\end{equation}
where
\[
S(d)=\left|\{(a_1,a_2)\in\mathcal{A}^2:a_1\ne a_2,\ d\mid a_1-a_2\}\right|.
\]

We now estimate \(S(d)\). 
By the definition of \(\mathcal{A}\), for \(a_1,a_2\in\mathcal{A}\),
\[
a_1-a_2=\left(g^{\lfloor k_1^{r_1}\rfloor}-g^{\lfloor\tilde k_1^{r_1}\rfloor}\right)
+\sum_{i=2}^{s-1}\left(g^{\lfloor k_i^{r_i}\rfloor}-g^{\lfloor\tilde k_i^{r_i}\rfloor}\right)
+\left(g^{\lfloor\lfloor k_s^{\lambda}\rfloor^{r_s}\rfloor}
-g^{\lfloor\lfloor\tilde k_s^{\lambda}\rfloor^{r_s}\rfloor}\right).
\]
Fix all parameters except \(k_1\). 
Since \(\sum_{i=1}^{s}1/r'_i=1\), the number of choices for the remaining \(2s-1\) parameters is
\[
\ll n^{1/r_1}\prod_{i=2}^{s-1}n^{2/r_i}\cdot n^{2/(\lambda r_s)}
=n^{2-1/r_1}.
\]
For fixed values of these parameters, the congruence \(d\mid a_1-a_2\) is equivalent to
\[
g^{\lfloor k_1^{r_1}\rfloor}\equiv g^{\lfloor\tilde k_1^{r_1}\rfloor}
+\sum_{i=2}^{s-1}\left(g^{\lfloor\tilde k_i^{r_i}\rfloor}
-g^{\lfloor k_i^{r_i}\rfloor}\right)
+\left(g^{\lfloor\lfloor\tilde k_s^{\lambda}\rfloor^{r_s}\rfloor}
-g^{\lfloor\lfloor k_s^{\lambda}\rfloor^{r_s}\rfloor}\right)\pmod d.
\]
If this congruence has no solution in \(k_1\), the contribution is zero.
Otherwise choose a solution \(k_0\). Since \((d,g)=1\), every solution satisfies
\[
g^{\lfloor k_1^{r_1}\rfloor}\equiv g^{\lfloor k_0^{r_1}\rfloor}\pmod d,
\]
that is,
\[
\lfloor k_1^{r_1}\rfloor\equiv\ell\pmod{m_d},
\]
where \(m_d=\operatorname{ord}_d(g)\) and \(\ell\) is the unique integer with \(0\le\ell<m_d\) and \(\ell\equiv\lfloor k_0^{r_1}\rfloor\pmod{m_d}\).
Therefore, for fixed remaining parameters, the number of possible \(k_1\) is at most
\[
N_d=\max_{0\le\ell<m_d}\left|\{k\in\mathbb N:\delta\lfloor k^{r_1}\rfloor<n,\
\lfloor k^{r_1}\rfloor\equiv\ell\pmod{m_d}\}\right|,
\]
and consequently
\begin{equation}\label{3.11}
S(d)\ll n^{2-1/r_1}N_d. 
\end{equation}

We now bound
\[
\sum_{\substack{(d,g)=1\\P^+(d)\le n}}\frac{\mu^2(d)}{|d|}N_d.
\]
By the function-field Mertens theorem \cite[Theorem 3]{rosen1999}, we
have
\begin{equation}\label{3.12}
\sum_{\substack{(d,g)=1\\P^+(d)\le n}}\frac{\mu^2(d)}{|d|}
=\prod_{\substack{|p|\le n\\p\nmid g}}\left(1+\frac1{|p|}\right)
\le\prod_{|p|\le n}\left(1-\frac1{|p|}\right)^{-1}\ll\log n,
\end{equation}
where the sums and products are over monic polynomials. Let
\(M_0=C_0n^{1-1/r_1}\log n\), where \(C_0=C_0(r_1,\delta)>0\) is chosen
sufficiently large. We claim that if \(m_d>M_0\), then
\begin{equation}\label{3.13}
N_d\le\frac{n^{1/r_1}}{\log n}.
\end{equation}
Indeed, suppose that \(N_d>n^{1/r_1}/\log n\). 
Since \(\delta\lfloor k^{r_1}\rfloor<n\) implies \(k\le C_1n^{1/r_1}\) for some \(C_1=C_1(\delta,r_1)>0\), the \(N_d\) admissible \(k\)'s lie in an interval of length at most \(C_1n^{1/r_1}\). 
Hence there exist two consecutive admissible \(k\)'s, say \(k'<k''=k'+L\), with
\[
L\le\frac{C_1n^{1/r_1}}{N_d-1}
<\frac{2C_1n^{1/r_1}}{N_d}<2C_1\log n.
\]
Since \(\lfloor k'^{r_1}\rfloor\equiv\lfloor k''^{r_1}\rfloor\equiv\ell\pmod{m_d}\) and \(r_1>1\), the positive difference \(\lfloor k''^{r_1}\rfloor-\lfloor k'^{r_1}\rfloor\) is a nonzero multiple of \(m_d\), so
\[
m_d\le\lfloor k''^{r_1}\rfloor-\lfloor k'^{r_1}\rfloor
\le k''^{r_1}-k'^{r_1}+1
\ll_{r_1,\delta}n^{1-1/r_1}L
\ll_{r_1,\delta}n^{1-1/r_1}\log n<M_0
\]
for \(C_0\) sufficiently large, contradicting \(m_d>M_0\). 
This proves (\ref{3.13}). 
Therefore the contribution from \(m_d>M_0\) is, by (\ref{3.12}) and (\ref{3.13}),
\[
\sum_{\substack{(d,g)=1\\P^+(d)\le n\\m_d>M_0}}\frac{\mu^2(d)}{|d|}N_d
\le\frac{n^{1/r_1}}{\log n}
\sum_{\substack{(d,g)=1\\P^+(d)\le n}}\frac{\mu^2(d)}{|d|}
\ll\frac{n^{1/r_1}}{\log n}\cdot\log n=n^{1/r_1}.
\]

It remains to treat the case \(m_d\le M_0\). For simplicity, we restrict ourselves to the case \(r_1\notin\mathbb Z\); the case \(r_1\in\mathbb Z\) can be treated similarly using Lemma~\ref{lem:cong}(i), and we omit the details. 
Applying Lemma~\ref{lem:cong}(ii) with \(r=r_1\), \(m=m_d\) and \(K=C_1n^{1/r_1}\) (so that \(K\asymp n^{1/r_1}\)), we obtain
\[
N_d\ll\frac{n^{1/r_1}}{m_d}
+n^{\frac{1-\eta}{r_1}}m_d^{-\alpha}+m_d^{1/r_1},
\]
where \(\eta=\eta(r_1)>0\) and \(\alpha=\alpha(r_1)>0\) are the constants from Lemma~\ref{lem:cong}(ii). 
We split the sum over \(m_d\le M_0\) into three parts:
\[
\sum_{\substack{(d,g)=1\\P^+(d)\le n\\m_d\le M_0}}\frac{\mu^2(d)}{|d|}N_d
\ll W_1+W_2+W_3,
\]
where
\[
W_1=n^{1/r_1}\sum_{\substack{(d,g)=1\\P^+(d)\le n\\m_d\le M_0}}
\frac{\mu^2(d)}{|d|\,m_d},
\qquad
W_2=n^{\frac{1-\eta}{r_1}}\sum_{\substack{(d,g)=1\\P^+(d)\le n\\m_d\le M_0}}
\frac{\mu^2(d)}{|d|\,m_d^{\alpha}},
\]
and
\[
W_3=\sum_{\substack{(d,g)=1\\P^+(d)\le n\\m_d\le M_0}}
\frac{\mu^2(d)}{|d|}\,m_d^{1/r_1}.
\]
By Lemma~\ref{lem:conv}, both series over all monic \(d\) coprime to \(g\) converge.
Hence
\[
W_1\le n^{1/r_1}\sum_{(d,g)=1}\frac{\mu^2(d)}{|d|\,m_d}\ll n^{1/r_1},
\]
and
\[
W_2\le n^{\frac{1-\eta}{r_1}}\sum_{(d,g)=1}\frac{\mu^2(d)}{|d|\,m_d^{\alpha}}
\ll n^{\frac{1-\eta}{r_1}}\ll n^{1/r_1}.
\]
For \(W_3\), since \(m_d\le M_0=C_0n^{1-1/r_1}\log n\), we have
\[
m_d^{1/r_1}\le M_0^{1/r_1}\ll n^{\frac1{r_1}-\frac1{r_1^2}}(\log n)^{1/r_1}.
\]
Using (\ref{3.12}), this gives
\[
W_3\ll n^{\frac1{r_1}-\frac1{r_1^2}}(\log n)^{1+1/r_1}\ll n^{1/r_1},
\]
the last step following from
\((\log n)^{1+1/r_1}=o(n^{1/r_1^2})\). Combining these estimates,
\begin{equation}\label{3.14}
\sum_{\substack{(d,g)=1\\P^+(d)\le n}}\frac{\mu^2(d)}{|d|}N_d
\ll n^{1/r_1}.
\end{equation}

Finally, by (\ref{3.11}) and (\ref{3.14}),
\[
\sum_{\substack{(d,g)=1\\P^+(d)\le n}}\frac{\mu^2(d)}{|d|}S(d)
\ll n^{2-1/r_1}
\sum_{\substack{(d,g)=1\\P^+(d)\le n}}\frac{\mu^2(d)}{|d|}N_d
\ll n^{2-1/r_1}\cdot n^{1/r_1}=n^2.
\]
Thus \(\mathcal{J}\ll q^n\). Since the contribution of \(\Delta=0\) is also \(\ll q^n\), we obtain \(\sum_fT(f)^2\ll q^n\). This proves (\ref{3.9}) and hence the theorem.
\end{proof}

\end{document}